\documentclass[12pt]{amsart}

\usepackage{amsmath,amssymb,amsbsy,amsfonts,amsthm,latexsym,
                        amsopn,amstext,amsxtra,euscript,amscd,mathrsfs,color,bm,cite,a4wide}
                       
\usepackage{float} 
\usepackage[english]{babel}
\usepackage{mathtools}
\usepackage{todonotes}
\usepackage{url}
\usepackage[colorlinks,linkcolor=blue,anchorcolor=blue,citecolor=blue,backref=page]{hyperref}

\usepackage{mathrsfs}

\usepackage{enumitem}

\def\le{\leqslant}
\def\ge{\geqslant}

\def\RS{\mathsf{RS}}

\def\TM{\mathsf{TM}}

\def\leq{\leqslant}

\usepackage{mathtools}
\usepackage{todonotes}
\usepackage[norefs,nocites]{refcheck}

\newtheorem{theorem}{Theorem}
\newtheorem{lemma}[theorem]{Lemma}

\newtheorem{cor}[theorem]{Corollary}

\newtheorem{rem}[theorem]{Remark}

\numberwithin{equation}{section}
\numberwithin{theorem}{section}
\numberwithin{table}{section}

\numberwithin{figure}{section}

\def \Grs{\cG_s(r)}

\def\\{\cr}
\def\({\left(}
\def\){\right)}
\def\[{\left[}
\def\]{\right]}
\def\<{\langle}
\def\>{\rangle}

\def\rf#1{\left\lceil#1\right\rceil}

\def\N{\mathbb{N}}
\def\cF{\mathcal F}

\def\cE{\mathcal E}
\def\cG{\mathcal G}
\def\cH{\mathcal H}
\def\cI{\mathcal I}
\def\cL{\mathcal L}

\def\e{{\mathbf{\,e}}}

\def\R{\mathbb{R}}

\def\mand{\qquad \mbox{and} \qquad}

\def\eqref#1{(\ref{#1})}

\begin{document}

\title{Weyl sums over integers with digital restrictions} 

 \author[I. E. Shparlinski] {Igor E. Shparlinski}
 
\address{Department of Pure Mathematics, University of New South Wales,
Sydney, NSW 2052, Australia}
\email{igor.shparlinski@unsw.edu.au}

 \author[J. M. Thuswaldner] {J{\"o}rg M. Thuswaldner}
 
\address{Department Mathematik und Informationstechnologie, Montanuniversit{\"a}t Leoben,   Le\-o\-ben, A-8700, Austria}
\email{joerg.thuswaldner@unileoben.ac.at}

\date{\today}

\begin{abstract}
We estimate Weyl sums over the integers with sum of binary digits either fixed or restricted by some congruence condition. In our proofs we use ideas that go back to a paper by Banks, Conflitti and the first author (2002). Moreover, we apply the ``main conjecture'' on the Vinogradov mean value theorem which has been established by Bourgain, Demeter and Guth (2016) as well as by Wooley (2016, 2019). We use our result to give an estimate of the discrepancy of point sets that are defined by the values of polynomials at arguments having the sum of binary digits restricted in different ways. 
\end{abstract}

\keywords{Sparse integers, sum of digits, Weyl sum}
\subjclass[2010]{11A63, 11L07}

\maketitle

\section{Introduction}
\subsection{Motivation and set-up} 
The study of the distribution of the $g$-ary sum of digits function in residue classes, using exponential sums, has been initiated  in a seminal paper by Gel'fond~\cite{Gelfond:68} (see~\cite{Fine:65} for a predecessor). Since then, exponential sum methods have been used by a vast number of authors in order to study distribution properties of the sum of digits function. In this area, the three conjectures at the end of Gel'fond's paper~\cite{Gelfond:68} have been a constant motivation. Two of them have been solved, the one on the joint distribution of sum of digits functions with respect to different bases in residue classes (see Kim~\cite{Kim:99}) and the one on the distribution of the sum of digits function on prime values in residue classes (see Mauduit and Rivat~\cite{MR:10}). Gel'fond's third conjecture on the distribution of the sum of digits function on polynomial values in residue classes is still open although there exist some important partial results, see {\it e.g.}~\cite{MR:09}. We also mention the work of Mauduit and S\'ark\"ozy~\cite{MS:97} that has initiated the active study of the arithmetic structure of sets of integers with fixed $g$-ary sum of digits (see, for example,~\cite{BaSh,DKDK,Emi1,Emi2,FM,MPS,MRS} and references therein). 
Exponential and character sums over integers with digital restrictions that are related to the results of the present paper have  been studied for instance in~\cite{BaCoSh,Emi3, OstShp, PfTh, ThTi}, and in particular applications to additive problems such as versions of the Waring problem have been given in~\cite{Emi3,PfTh, ThTi, Wag}.

In the present paper we build upon the ideas of~\cite{BaCoSh, OstShp} and provide estimates for Weyl sums whose range is restricted to integers with various conditions on their sum of digits function, for example,  such as integers  with  fixed sum of digits. As in the case of classical Weyl sums, an important ingredient in our proofs is the version of the Vinogradov mean value theorem established by Bourgain, Demeter and Guth~\cite{BDG} and Wooley~\cite{Wool1,Wool2}. 

Although all our methods work in much larger generality, for convenience and to avoid unnecessary combinatorial complications (external to our method),  we restrict our attention to the binary sum of digits function.

Let $\sigma(n)$ denote the sum of binary digits of $n$,  that is,
$$\sigma(n)=\sum_{j\ge 0}a_j(n),$$
where 
\begin{equation}\label{eq:binrep}
n=\sum_{j\ge 0} a_j(n) 2^j, \qquad a_j(n)\in\{0,1\}.
\end{equation}

For $r,m\in \N$ and $k\in\{0,\ldots,m-1\}$ let
$$
\cE_{k,m}(r) = \{n < 2^r :~ \sigma(n)\equiv k\pmod{m}\}
$$
be the set of integers with $r$ digits whose sum of digits is congruent to $k$ modulo $m$.
Moreover, for any integers $0\le s\le r$, let
$$
 \Grs=\{0 \leq n< 2^r:~\sigma(n)=s\}.
$$ 
Then $ \Grs$ is the set of integers with $r$ digits (in base $2$) such that the sum of the digits is equal to $s$.

As mentioned before, various bounds of exponential and character sums over the integers from the sets $\cE_{k,m}(r)$, $ \Grs$ and some other sets of integers with restricted digits have been considered in~\cite{BaCoSh,FrSh,OstShp}. Here we obtain 
new bounds of {\it Weyl sums} restricted by these sets. In particular,  one can use our method to improve some of the previous results on the variants of Waring's problem  studied in~\cite{Emi3,PfTh, ThTi}.

More precisely, let
$$
\e(z) =  \exp(2 \pi i z),
$$
and   let
\begin{equation}
\label{eq:Poly}
f(Z) = \sum_{i=1}^d \alpha_i Z^i \in \R[Z]
\end{equation}
be a polynomial.  Then the main objective of the present paper is to estimate the Weyl sums
$$
U_f(r,\ell,k,m) = \sum_{n \in \cE_{k,m}(r)}  \e(\ell f(n))
\mand 
S_{f}(r, \ell, s) = \sum_{n \in  \Grs}\e(\ell f(n)), 
$$
as well as Weyl sums twisted by special sequences related to the binary sum of digits function $\sigma$, like the Thue--Morse sequence and the Rudin--Shapiro sequence. From these estimates we  derive equidistribution results.

\subsection{Outline}
The paper is organised as follows. 

In Section~\ref{sec:Main} we formulate our main results giving bounds on the sums
$U_f(r,\ell,k,m)$ and $S_{f}(r, \ell, s)$ together with their applications. 

In Section~\ref{sec:Preps} we give some preparatory results on unrestricted Weyl sums with exponential of the form $\e\(hf(n)\)$, $n =1, \ldots, N$, which are uniform in 
$h$. Although results of this type based on previous versions of the Vinogradov mean value 
theorem have been known~\cite[Theorem~I, Chapter~VI]{Vin}, our result, which is based on the modern
advances~\cite{BDG,Wool1,Wool2} seems to be new and  fills a gap in the literature on this subject. 

Section~\ref{sec:Thue} contains the proofs of our results on the Weyl sums $U_f(r,\ell,k,m)$ whose range is restricted by sum of digits congruences and Weyl sums twisted by the Thue-Morse and the Rudin-Shapiro sequence. Our main tools include the ideas of~\cite{BaCoSh} and the uniform estimates of the classical Weyl sums provided in Section~\ref{sec:Preps}.

In Section~\ref{sec:sparseproof} we prove our results on the Weyl sums $S_{f}(r, \ell, s)$
ranging over sets with fixed sum of digits. Again we use~\cite{BaCoSh} and the estimates of Section~\ref{sec:Preps}. 

Finally, in Section~\ref{sec:equi} we employ the Erd\H{o}s--Tur{\'a}n--Koksma inequality to establish our equidistribution results.

\section{Main results} 
\label{sec:Main} 

\subsection{Concepts and notation used in the main results}  According to Dirichlet's approximation theorem there exist infinitely many  pairs $(a,q)$ of integers with $q \ge 1$ and $\gcd(a,q) = 1$ satisfying
\begin{equation}
\label{eq:Approx}
\left|\alpha_d -\frac{a}{q} \right | < \frac{1}{q^2}.
\end{equation} 
For an irrational number $\alpha$, we define its {\it Diophantine type}
$\tau$ by the relation
$$
\tau=\sup\left\{\vartheta\in\R~:~\liminf_{q\to\infty,~q\in\N}
q^\vartheta\,\|\alpha q\|=0\right\}, 
$$
where  
$$\|x\|=\min\{x-n:~n\in\mathbb{Z}\}
$$ 
denotes the distance between $x\in \R$ and the closest integer. The celebrated results of Khinchin~\cite{Khin} and of Roth~\cite{Roth} assert that $\tau=1$ for almost all reals (in the sense of the Lebesgue measure) and all irrational algebraic numbers $\alpha$, respectively; see also~\cite{Bug,Schm}. Clearly, if  $\alpha$ is of Diophantine type $\tau=1$ then the partial quotients of its continued fraction expansion cannot get too big and, hence, the integers $q \ge 1$ satisfying~\eqref{eq:Approx} for some $a\in\mathbb{Z}$ are well-spaced. 

In the sequel we de denote the fractional part of a real number $y$ by $\{y\}$.  Let $M \subset \N$ be finite and let $f\in \mathbb{R}[Z]$ be given. Then the {\em discrepancy} of the set
$$
\{ \{f(n)\} :~ n\in M  \}
$$
is defined by  (see, for example,~\cite[Equation~(1.11)]{DT:97})
$$
D_f(M) = \sup_{\cI \subseteq [0,1)} \left|
\frac{\#\left\{ n \in M : ~\{f(n)\} \in \cI \right\}}{\# M} - \cL(\cI)
\right|,
$$
where $\cL$ denotes the Lebesgue measure on $[0,1)$.

We recall that the expressions $A(x) \ll B(x)$ and $A(x)=O(B(x))$ are each equivalent to the
statement that $|A(x)|\le cB(x)$ for some constant $c>0$ and for all $x$ larger than some threshold value $x_0 > 0$.  Throughout the paper, the implied constant $c$ in the symbols ``$O$'' and ``$\ll$'' may depend on the positive integer  
parameters $m$ and  $\nu$ (and, where obvious, the real parameter  $\varepsilon> 0$). 

We also write $A(r)^{o(1)}$ for any function $B(r)$ such for any fixed $\varepsilon> 0$ we have $A(r)^{-\varepsilon}  \ll B(r) \ll A(r)^{\varepsilon}$ as $r\to \infty$. In particular, 
$$
 2^{-\varepsilon r}\ll 2^{o(r)} \ll 2^{\varepsilon r} \mand  2^{(1- \varepsilon) r}  \ll 2^{\(1+ o(1)\)r } \ll 2^{(1+\varepsilon) r}  
$$
for any fixed $\varepsilon> 0$.

\subsection{Weyl sums with congruence conditions on the  sum of digits}\label{sec:cong}

We start with the sums $U_f(r,\ell,k,m)$.  We recall that the implied
constants can depend on $m$, however it is easy to make our results uniform with respect to $m$ as well. 

\begin{theorem}
\label{thm:Weyl-cong}
For  any polynomial $f\in \R[Z]$ of degree $d \ge 3$ of the form~\eqref{eq:Poly} and
with the leading coefficient $\alpha_d$ satisfying~\eqref{eq:Approx}, 
we have 
\begin{equation}\label{eq:Weyl-cong}
|U_f(r,\ell,k,m) |  \le       2^{(1+o(1))r} \(\ell^{\eta_1(d)}q^{-\eta_1(d)} + \ell^{\vartheta(d)}2^{-\zeta_1(d)  r} + 2^{-\zeta_2(d)  r} +  2^{-\zeta_3(d)  r} q^{\eta_2(d) } \)^{1/2}
\end{equation} 
as $r\to \infty$, where 
\begin{equation}\label{eq:etazetatheta}
\begin{split}
& \eta_1(d)  =  \frac{1}{d^2-2d+2},  \qquad   \eta_2(d)  = \frac{1}{(d-1)^2},
 \qquad \vartheta(d)=\frac{1}{d(d-1)},\\
& \  \zeta_1(d)   = \frac{d-2}{d(d-1)}, \qquad \zeta_2(d)  = \frac{1}{d^2-3d+3},  \qquad  \zeta_3(d)   = \frac{1}{d-1}. 
\end{split}
\end{equation}
\end{theorem}

Note that 
\begin{equation}\label{eq:zetaineq}
\zeta_1(3) < \zeta_2(3) \mand \zeta_1(d) > \zeta_2(d), \quad \text{for}\ d\ge 4.
\end{equation}
 Thus, for $d=3$ we can omit the term $2^{-\zeta_2(d)  r}$ in~\eqref{eq:Weyl-cong}, while for $d\ge 4$ and $\ell=1$ the term $2^{-\zeta_1(d)  r}$ can be omitted in~\eqref{eq:Weyl-cong}.

If $\alpha_d$ is of Diophantine type $\tau=1$, for given $r$
we can choose $q$ in~\eqref{eq:Approx} 
in a way that 
\begin{equation}\label{eq:qchoice}
q = \ell^{\eta_1(d)/(\eta_1(d)+\eta_2(d))}2^{\(\zeta_3(d)/(\eta_1(d)+\eta_2(d)) + o(1)\) r}
\end{equation}
(note that we do not need to multiply by a factor $\ell^{o(1)}$ because this is absorbed in $2^{o(r)}$ as we can always assume 
$\log \ell \ll r$ as otherwise the bound is trivial). Indeed, this choice of $q$ optimises the bound of Theorem~\ref{thm:Weyl-cong} and the expression in the brackets on the right hand side of~\eqref{eq:Weyl-cong} becomes
\begin{equation}\label{eq:2pwr}
\ell^{\eta_1(d)\eta_2(d)/(\eta_1(d)+\eta_2(d))}2^{-(\eta_1(d) \zeta_3(d)/(\eta_1(d)+\eta_2(d)) + o(1))r} + \ell^{\vartheta(d)}2^{-\zeta_1(d)r}  + 2^{-\zeta_2(d)r}.
\end{equation}
Thus Theorem~\ref{thm:Weyl-cong} implies the following result.

\begin{cor} 
\label{cor:Weyl-cong-Dioph-1}  For  any polynomial $f\in \R[Z]$ of degree $d \ge 3$ of the form~\eqref{eq:Poly} and 
with the leading coefficient $\alpha_d$ of Diophantine type $\tau=1$, 
we have 
\begin{equation}\label{eq:cor2bdcong}
|U_{f}(r,\ell,k,m)|   \le       2^{(1+o(1))r} \(\(\frac{\ell^{\eta_2(d)}}{2^{\zeta_3(d)r}}\)^{\eta_1(d)/(\eta_1(d)+\eta_2(d))} + \frac{\ell^{\vartheta(d)}}{2^{\zeta_1(d)r}}  + 2^{-\zeta_2(d)r}\)^{1/2}
\end{equation}
as $r\to \infty$, where $\eta_1(d), \eta_2(d), \vartheta(d),\zeta_1(d), \zeta_2(d)$ and $\zeta_3(d)$ are defined in~\eqref{eq:etazetatheta}.
\end{cor}  

Direct calculation shows that 
$$
\frac{\eta_1(d) \zeta_3(d)}{\eta_1(d)+\eta_2(d)}> \min\{\zeta_1(d),\zeta_2(d)\}
$$
holds for $d\ge 3$. Thus for $\ell=1$ we define 
$$
U_{f}(r,k,m) = U_{f}(r,1,k,m)
$$ 
and observe that in this case the first summand in~\eqref{eq:2pwr} never dominates and, hence, Corollary~\ref{cor:Weyl-cong-Dioph-1} has the following much simpler bound for  $\ell=1$. 

\begin{cor}
\label{cor:Weyl-cong-Dioph-simpl}  For  any polynomial $f\in \R[Z]$ of degree $d \ge 3$ of the form~\eqref{eq:Poly} and 
with the leading coefficient $\alpha_d$ of Diophantine type $\tau=1$, 
we have 
\begin{equation}\label{eq:cor2bd1cong}
|U_{f}(r,k,m)|   \le  2^{(1-\xi(d)/2  +o(1))r}   
\end{equation}
as $r\to \infty$, where 
\begin{equation}\label{eq:xidef}
\xi(d)=\min\{\zeta_1(d),\zeta_2(d) \}.
\end{equation}
\end{cor}

As mentioned above, $\xi(3)= \zeta_1(3)$ and $\xi(d)= \zeta_2(d)$ holds for $d\ge 4$.

Using Corollary~\ref{cor:Weyl-cong-Dioph-1} we  are able to prove an estimate on the discrepancy of the point set 
$$
\left\{ \{f(n)\}  : ~n \in \cE_{k,m}(r)\right\}.
$$
Our result reads as follows.
 
\begin{theorem}
\label{thm:cong-Equidistr}
For  any polynomial $f\in \R[Z]$ of degree $d \ge 3$ of the form~\eqref{eq:Poly} and with the leading coefficient $\alpha_d$ of Diophantine type $\tau=1$, 
we have 
\begin{equation}\label{eq:thequicong}
D_f(\cE_{k,m}(r))  \le  2^{-\min\{\nu_1(d), \nu_2(d)\} r + o(r)}
\end{equation} 
as $r\to \infty$, where 
$$
\nu_1(d)  = \frac{d-2}{2d^2-2d+1} \qquad\hbox{and}\qquad   \nu_2(d)  = \frac{1}{2d^2-6d+6}. 
$$
\end{theorem}

Direct calculation shows that 
$$
\nu_1(3) < \nu_2(3) \mand \nu_1(d) > \nu_2(d), \quad \text{for}\ d\ge 4.
$$

The estimate in Theorem~\ref{thm:Weyl-cong} becomes trivial if one the summands on the right hand side of~\eqref{eq:Weyl-cong} is too large. This has the effect that it is trivial  if $q$ is small. We circumvent this problem  in the following result, which generally provides a slightly weaker estimate than the one in Theorem~\ref{thm:Weyl-cong} but instead 
 remains meaningful  for smaller $q$.

\begin{theorem}
\label{thm:Weyl-congLog}
For  any polynomial $f\in \R[Z]$ of degree $d \ge 3$ of the form~\eqref{eq:Poly} and
with the leading coefficient $\alpha_d$ satisfying~\eqref{eq:Approx}, 
we have 
\begin{equation}\label{eq:Weyl-congLog}
|U_{f}(r,\ell,k,m)|  \le       r2^{r} \(\ell^{\widetilde\eta_1(d) } q^{-\widetilde\eta_1(d) } + \ell^{\widetilde\vartheta(d)}2^{-\widetilde\zeta_1(d)  r} + 2^{-\widetilde\zeta_2(d)  r} +  2^{-\widetilde\zeta_3(d)  r} q^{\widetilde\eta_2(d) } \)^{1/2}
\end{equation}
as $r\to \infty$, where 
\begin{equation}\label{eq:tildes}
\begin{split}
&  \widetilde\eta_1(d) = \frac{1}{d^2-2d+4},  \qquad   \widetilde\eta_2(d)   = \frac{1}{d^2-2d+3},  \qquad \widetilde\vartheta(d) = \frac{1}{d^2-d+2},\\
& \widetilde\zeta_1(d)   = \frac{d-2}{d^2-d+2}, \qquad \widetilde\zeta_2(d) = \frac{1}{d^2-3d+5},  \qquad  \widetilde\zeta_3(d)   = \frac{d-1}{d^2-2d+3}. 
\end{split} 
\end{equation}
\end{theorem}

Analogously to~\eqref{eq:zetaineq} we have 
\begin{equation}\label{eq:tildeineq}
\widetilde\zeta_1(3) < \widetilde\zeta_2(3) \mand \widetilde\zeta_1(d) > \widetilde\zeta_2(d)\quad \text{for}\ d\ge 4.
\end{equation}
Thus for $d=3$ we can omit the term $2^{-\widetilde\zeta_2(d)  r}$ in~\eqref{eq:Weyl-congLog}, while for $d\ge 4$ and $\ell=1$ the term $2^{-\widetilde\zeta_1(d)  r}$ can be omitted in~\eqref{eq:Weyl-congLog}.

\subsection{Weyl sums twisted with  digital sequences} 
\label{sec:WS-twist} 

We start by twisting Weyl sums by the \emph{Thue--Morse sequence} 
$t_n = \((-1)^{\sigma(n)}\)_{n\in\N}$, that is, we consider
$$
\TM_f(r,\ell) = \sum_{n<2^r} t_n \e(\ell f(n)).
$$
Observing that Theorem~\ref{thm:Weyl-cong}, taken with $m=2$, immediately implies that   the sums $\TM_f(r,\ell)$ satisfy the same bound as the sums $U_f(r,\ell,k,m)$, that is
$$
|\TM_f(r,\ell)|  \le       2^{(1+o(1))r} \(\ell^{\eta_1(d)}q^{-\eta_1(d)} + \ell^{\vartheta(d)}2^{-\zeta_1(d)  r} + 2^{-\zeta_2(d)  r} +  2^{-\zeta_3(d)  r} q^{\eta_2(d) } \)^{1/2}
$$
as $r\to\infty$, with $\eta_1(d), \eta_2(d), \vartheta(d),\zeta_1(d), \zeta_2(d)$ and $\zeta_3(d)$ as in~\eqref{eq:etazetatheta}.

Next we consider  Weyl sums twisted by the  \emph{Rudin--Shapiro sequence}
$\rho_n = \((-1)^{\chi_{11}(n)}\)_{n\in\N}$
$$
\RS_f(r,\ell) = \sum_{n<2^r} \rho_n \e(\ell f(n)), 
$$
where $\chi_{11}(n)$ is the number of (possibly overlapping) occurrences of the block $11$ in the binary expansion~\eqref{eq:binrep} of $n \in \N$.

\begin{theorem}
\label{thm:Weyl-RS}
For  any polynomial $f\in \R[Z]$ of degree $d \ge 3$ of the form~\eqref{eq:Poly} and
with the leading coefficient $\alpha_d$ satisfying~\eqref{eq:Approx}, 
we have 
$$
|\RS_{f}(r,\ell)|   \le    2^{(1+o(1))r} \(\ell^{\eta_1(d)}q^{-\eta_1(d)} + \ell^{\vartheta(d)}2^{-\zeta_1(d)  r} + 2^{-\zeta_2(d)  r} +  2^{-\zeta_3(d)  r} q^{\eta_2(d) } \)^{1/2}
$$
as $r\to \infty$, where $\eta_1(d), \eta_2(d), \vartheta(d),\zeta_1(d), \zeta_2(d)$ and $\zeta_3(d)$ are defined in~\eqref{eq:etazetatheta}.
\end{theorem}

To show the power and agility of our approach, next we consider the sums
$$
W_f(r,\ell)  = \sum_{n<2^r} t_n t_{n+1}  \e(\ell f(n)).
$$
with double twist by the Thue--Morse sequence. 

\begin{theorem}
\label{thm:Weyl-sigmasigma}
For  any polynomial $f\in \R[Z]$ of degree $d \ge 3$ of the form~\eqref{eq:Poly} and
with the leading coefficient $\alpha_d$ satisfying~\eqref{eq:Approx}, 
we have 
$$
|W_f(r,\ell) |    \le     2^{(1+o(1))r} \(\ell^{\eta_1(d)}q^{-\eta_1(d)} + \ell^{\vartheta(d)}2^{-\zeta_1(d)  r} + 2^{-\zeta_2(d)  r} +  2^{-\zeta_3(d)  r} q^{\eta_2(d) } \)^{1/2}
$$
as $r\to \infty$, where $\eta_1(d), \eta_2(d), \vartheta(d),\zeta_1(d), \zeta_2(d)$ and $\zeta_3(d)$ are defined in~\eqref{eq:etazetatheta}.\end{theorem}

Note that all the variants of Theorem~\ref{eq:Weyl-cong} stated in Section~\ref{sec:cong} can be established for the results of the present section by similar arguments.

\subsection{Weyl sums over integers with a fixed sum of digits} 

Our first main result on Weyl sums over sets with fixed binary sum of digits reads as follows.

\begin{theorem}
\label{thm:Weyl-sparse}
For  any polynomial $f\in \R[Z]$ of degree $d \ge 3$ of the form~\eqref{eq:Poly} and
with the leading coefficient $\alpha_d$ satisfying~\eqref{eq:Approx}, 
we have 
\begin{equation}\label{eq:Weyl-sparse}
\begin{split} 
|S_{f}(r, \ell, s)|   \le       2^{(1/2+o(1))r} & \binom{r}{s}^{1/2} \\
&\(\ell^{\eta_1(d)}q^{-\eta_1(d)} + \ell^{\vartheta(d)}2^{-\zeta_1(d)  r} + 2^{-\zeta_2(d)  r} +  2^{-\zeta_3(d)  r} q^{\eta_2(d) } \)^{1/2}
\end{split}
\end{equation}
as $r\to \infty$, where $\eta_1(d), \eta_2(d), \vartheta(d),\zeta_1(d), \zeta_2(d)$ and $\zeta_3(d)$ are defined in~\eqref{eq:etazetatheta}.
\end{theorem}

Analogously to Theorem~\ref{thm:Weyl-cong}, by~\eqref{eq:zetaineq} we can omit the term $2^{-\zeta_2(d)  r}$ in~\eqref{eq:Weyl-sparse} for $d=3$, while for $d\ge 4$ and $\ell=1$ the term $2^{-\zeta_1(d)  r}$ can be omitted in~\eqref{eq:Weyl-sparse}.

As in Section~\ref{sec:cong} we further study the case where $\alpha_d$ is of Diophantine type $\tau=1$. As above, for given $r$ we can choose $q$ (given as in~\eqref{eq:Approx}) as in~\eqref{eq:qchoice} which optimises the bound of Theorem~\ref{thm:Weyl-sparse}. For this choice of $q$ the expression in the brackets on the right hand side of~\eqref{eq:Weyl-sparse} is given by~\eqref{eq:2pwr}. Thus Theorem~\ref{thm:Weyl-sparse} implies the following result.

\begin{cor} 
\label{cor:Weyl-sparse-Dioph-1}  For  any polynomial $f\in \R[Z]$ of degree $d \ge 3$ of the form~\eqref{eq:Poly} and 
with the leading coefficient $\alpha_d$ of Diophantine type $\tau=1$, 
we have 
\begin{equation}\label{eq:cor2bd}
|S_{f}(r, \ell, s)|   \le       2^{(1/2+o(1))r}   \binom{r}{s}^{1/2} \(\(\frac{\ell^{\eta_2(d)}}{2^{\zeta_3(d)r}}\)^{\eta_1(d)/(\eta_1(d)+\eta_2(d))} + \frac{\ell^{\vartheta(d)}}{2^{\zeta_1(d)r}}  + 2^{-\zeta_2(d)r}\)^{1/2}
\end{equation}
as $r\to \infty$, where $\eta_1(d), \eta_2(d), \vartheta(d),\zeta_1(d), \zeta_2(d)$ and $\zeta_3(d)$ are defined in~\eqref{eq:etazetatheta}.
\end{cor}  

For $\ell=1$ we define 
$$
S_{f}(r,s) = S_{f}(r,1,s)
$$ 
and observe that by the same argument as in Section~\ref{sec:cong}, Corollary~\ref{cor:Weyl-sparse-Dioph-1} has the following much simpler bound for  $\ell=1$. 

\begin{cor}
\label{cor:Weyl-sparse-Dioph-simpl}  For  any polynomial $f\in \R[Z]$ of degree $d \ge 3$ of the form~\eqref{eq:Poly} and 
with the leading coefficient $\alpha_d$ of Diophantine type $\tau=1$, 
we have 
\begin{equation}\label{eq:cor2bd1}
|S_{f}(r,s)|   \le    \binom{r}{s}^{1/2}   2^{(1-\xi(d)  +o(1))r/2}   
\end{equation}
as $r\to \infty$, where $\xi(d)$ is defined in~\eqref{eq:xidef}.
\end{cor}

By~\eqref{eq:zetaineq}, $\xi(3)= \zeta_1(3)$ and $\xi(d)= \zeta_2(d)$ holds for $d\ge 4$.

Contrary to the bound~\eqref{eq:cor2bd1cong} in Corollary~\ref{cor:Weyl-cong-Dioph-simpl},
due to the additional parameter $s$ it is {\it a priori} not clear when the bound~\eqref{eq:cor2bd1} is smaller than the trivial bound $\# \Grs$. In order to clarify this, recall that 
\begin{equation}\label{eq:GHrs}
\# \Grs=\binom{r}{s} = 2^{rH(s/r)+ o(r)},
\end{equation}
where 
\begin{equation}\label{eq:Hrs}
H(\gamma) = \frac{- \gamma \log \gamma  -  (1-\gamma) 
\log (1-\gamma)}{\log 2}  
\end{equation}
denotes the {\it binary entropy function\/}, see, for example,~\cite[Section~10.11]{MS}.
Since $H(\gamma)  \to 1$ as $\gamma\to 1/2$ we see that under the conditions of 
Corollary~\ref{cor:Weyl-sparse-Dioph-1} we have a nontrivial bound for sums over integers with rather sparse binary representations.

\begin{cor}
\label{cor:Weyl-sparse-Dioph-2} 
For any $d\ge 3$ there are constants $\rho(d) <1/2$ and $\omega(d)  > 0$ depending only on $d$ such that
for  any polynomial $f\in \R[Z]$ of degree $d \ge 3$ of the form~\eqref{eq:Poly}   and 
with the leading coefficient $\alpha_d$ of Diophantine type $\tau=1$, 
for $s/r \in  [\rho(d), 1/2]$, 
we have 
\begin{equation}\label{eq:rhoomega}
|S_{f}(r,s)|   \le      \binom{r}{s}^{1-\omega(d)}
\end{equation}
provided that $r$ is large enough. 
\end{cor}

In view of~\eqref{eq:cor2bd1}, \eqref{eq:GHrs} and~\eqref{eq:Hrs} we get a nontrivial bound 
in~\eqref{eq:rhoomega} for large $r$ if and only if 
\begin{equation}\label{eq:xiHbound}
1-\xi(d) < H(\rho(d))
\end{equation}
holds. We calculated threshold values for $\rho(d)$ for $d\in\{3,\ldots, 10\}$. In particular, if $\rho(d) \ge \rho_0(d)$ with $\rho_0(d)$ given in Table~\ref{tab:rho}, 
then the quantity $\omega(d)$ in~\eqref{eq:rhoomega} can be chosen positively.

 \newcommand{\fontsmall}{\fontsize{11pt}{12pt}\selectfont}
\begin{table}[H]
\begin{minipage}[c]{6.2in}
\fontsmall
\begin{tabular}{|c||c|c|c|c|c|c|c|c|}
\hline
$d$ & $3$ & $4$& $5$& $6$& $7$& $8$& $9$& $10$ \\
\hline
$\rho_0(d)$ & $0.264414$ & $0.281247$ & $0.338192$ & $0.372247$ & $0.394662$ & $0.410466$ & $0.422184$ & $0.431208$ \\
\hline
\end{tabular} 
\end{minipage}
\caption{The threshold values $\rho_0(d)$ for $\rho(d)$ for $3\le d \le 10$. \label{tab:rho}}
\end{table}

Using Corollary~\ref{cor:Weyl-sparse-Dioph-1} we will establish the following estimate on the discrepancy of the point set 
$$
\left\{ \{f(n)\}  : ~n \in \Grs\right\}.
$$

\begin{theorem}
\label{thm:Sparse-Equidistr}
For  any polynomial $f\in \R[Z]$ of degree $d \ge 3$ of the form~\eqref{eq:Poly} and with the leading coefficient $\alpha_d$ of Diophantine type $\tau=1$, 
we have 
\begin{equation}\label{eq:thequi}
D_f( \Grs)  \le   {r \choose s}^{-\beta_1(d)}
2^{\gamma_1(d) r + o(r)}
+ 
{r \choose s}^{-\beta_2(d)}
2^{\gamma_2(d) r + o(r)}
+
{r \choose s}^{-\beta_3(d)}
2^{\gamma_3(d) r + o(r)}
\end{equation} 
as $r\to \infty$, where 
\begin{align*}
&\qquad\quad \beta_1(d)  = \frac{2 d^2-4 d+3}{4d^2-8d+7},  \qquad   \beta_2(d)  = \frac{d(d-1)}{2d^2 -2d+1}, \qquad \beta_3(d)=\frac{1}{2},\\
& \gamma_1(d)   =\frac{2 d^2-5 d+4}{4d^2-8d+7}, \qquad \gamma_2(d)  = \frac{d^2-2d+2}{2 d^2-2d+1},  \qquad  \gamma_3(d)   = \frac{(d-1) (d-2)}{2 (d^2-3d+3)}. 
\end{align*}
\end{theorem}

It follows from the proofs of Theorems~\ref{thm:cong-Equidistr} and~\ref{thm:Sparse-Equidistr} (or by direct calculation) that $\nu_1=\gamma_2-\delta_2$, and $\nu_2=\gamma_3-\delta_3$ (since $\gamma_1-\delta_1 > \min\{\gamma_2-\delta_2,\gamma_3-\delta_3\}$ for all $d\ge 3$, the exponent $\gamma_1-\delta_1$ does not occur on the right hand side of the estimate~\eqref{eq:thequicong}).

Note that for $d\ge 3$ we always have $\beta_j(d)  > \gamma_j(d)$ ($j\in\{1,2,3\}$). Because 
$$
\# \Grs = {r \choose s} \gg 2^r r^{-1/2}
$$ 
for $s$ close to $r/2$, the estimate in~\eqref{eq:thequi} is certainly nontrivial for these choices of $s$. Thus, analogously to Corollary~\ref{cor:Weyl-sparse-Dioph-2} we can formulate the following result.

\begin{cor}
\label{cor:equi1} 
Let $d\ge 3$ and $\rho(d)$ as in Corollary~\ref{cor:Weyl-sparse-Dioph-2}. Let $f\in \R[Z]$ be any polynomial of degree $d \ge 3$ of the form~\eqref{eq:Poly} and with leading coefficient $\alpha_d$ of Diophantine type $\tau=1$. Then there is $\mu(d)>0$ such that for $s/r \in  [\rho(d), 1/2]$ we have 
\begin{equation}\label{eq:disspec}
D_f( \Grs)  \le   \binom{r}{s}^{-\mu(d)}
\end{equation}
provided that $r$ is large enough. 
\end{cor}

In view of~\eqref{eq:GHrs}, \eqref{eq:Hrs} and~\eqref{eq:thequi} we get a nontrivial bound 
in~\eqref{eq:disspec} for large $r$ if and only if 
$$
\max\{ \gamma_j(d) / \beta_j(d):~1\le j\le 3\}  < H(\rho(d)).
$$
Direct calculation shows that 
$$
\max\{ \gamma_j(d) / \beta_j(d):~ 1\le j\le 3\} = 1-\xi(d)
$$
holds for $d\ge3$.
Thus~\eqref{eq:xiHbound} implies that the function $\rho(d)$ in Corollary~\ref{cor:equi1} can indeed chosen to be the same as in Corollary~\ref{cor:Weyl-sparse-Dioph-2}.

As in Theorem~\ref{thm:Weyl-cong}, the estimate in Theorem~\ref{thm:Weyl-sparse} becomes trivial if one of the summands on the right hand side of~\eqref{eq:Weyl-sparse} is too large. Thus again we give a result that is valid for a wider range of $q$.

\begin{theorem}
\label{thm:Weyl-sparseLog}
For  any polynomial $f\in \R[Z]$ of degree $d \ge 3$ of the form~\eqref{eq:Poly} and
with the leading coefficient $\alpha_d$ satisfying~\eqref{eq:Approx}, 
we have 
\begin{equation}\label{eq:Weyl-sparseLog}
|S_{f}(r,\ell, s)|   \le       r2^{r/2}   \binom{r}{s}^{1/2} \(\ell^{\widetilde\eta_1(d) } q^{-\widetilde\eta_1(d) } + \ell^{\widetilde\vartheta(d)}2^{-\widetilde\zeta_1(d)  r} + 2^{-\widetilde\zeta_2(d)  r} +  2^{-\widetilde\zeta_3(d)  r} q^{\widetilde\eta_2(d) } \)^{1/2}
\end{equation}
as $r\to \infty$, where  $\widetilde\eta_1(d), \widetilde\eta_2(d), \widetilde\vartheta(d),\widetilde\zeta_1(d), \widetilde\zeta_2(d)$ and $\widetilde\zeta_3(d)$ are defined in~\eqref{eq:tildes}.
\end{theorem}  

As in Theorem~\ref{thm:Weyl-congLog}, by~\eqref{eq:tildeineq}, for $d=3$ we can omit the term $2^{-\widetilde\zeta_2(d)  r}$ in~\eqref{eq:Weyl-sparseLog}, while for $d\ge 4$ and $\ell=1$ the term $2^{-\widetilde\zeta_1(d)  r}$ can be omitted in~\eqref{eq:Weyl-sparseLog}.

\section{Preparations}
\label{sec:Preps}

\subsection{Optimization of power sums}
We need the following technical result, see~\cite[Lemma~2.4]{GrKol}.

\begin{lemma}
\label{lem: Optim}
For $I,J\in\mathbb{N}$ let
$$
F(Z)=\sum_{i=1}^I A_i Z^{a_i}+\sum_{j=1}^JB_j Z^{-b_j},
$$
where $A_i,B_j,a_i$ and $b_j$ are positive for $1\le i\le I$ and $1\le j \le J$. Let $0\leq Z_1\leq Z_2$. Then there is some $Z\in [Z_1,Z_2]$ with
$$
F(Z)\ll \sum_{i=1}^I\sum_{j=1}^J
\(A_i^{b_j}B_j^{a_i}\)^{1/(a_i+b_j)}+\sum_{i=1}^I A_iZ_1^{a_i}+\sum_{j=1}^J B_jZ_2^{-b_j},
$$
where the implied constant depends only on $I$ and $J$.
\end{lemma}

\subsection{Bounds of Weyl sums}
In order to prove Theorem~\ref{thm:Weyl-sparse} we need to give a bound of the classical Weyl sums with integer multiples of real polynomials.  Namely, for an integer $h$ we we need to estimate the sum 
$$
T_{f}(h,N)=\sum_{n =1}^N \e(hf(n)).
$$
We  obtain a bound for $T_f$ in terms of the leading coefficient of $f\in \R[Z]$ if $f$ is given as in~\eqref{eq:Poly} and~\eqref{eq:Approx}. Our main tool  is the {\it Vinogradov mean value theorem\/}. More precisely, we set 
$$
J_{d,s}(N) = \int_{0}^1\dots \int_{0}^1  \left| \sum_{n =1}^N \e(\alpha_1n + \dots +\alpha_d n^d) \right|^{2s} \mathrm{d}\alpha_1\ldots \mathrm{d}\alpha_d 
$$
and recall the optimal form of the Vinogradov mean value theorem established by Bourgain, Demeter and Guth~\cite{BDG} and Wooley~\cite{Wool1,Wool2}. This result states that, for $s \ge 1$ and $d \ge 2$, 
\begin{equation}
\label{eq:MVT}
J_{d,s}(N)\ll  N^{s +o(1)}   +  N^{2s - s(d)+o(1)}, 
\end{equation} 
where  
$$
s(d)=d(d+1)/2
$$
is the so-called {\it critical exponent}. Note that the first or second summand on the right hand side of~\eqref{eq:MVT} dominate if $s \le s(d)$ or $s\ge s(d)$, respectively.
Furthermore, by~\cite[Corollary~1.3]{Wool2} for $d\ge 3$ and $s > s(d)$ we can sharpen~\eqref{eq:MVT} to the asymptotic formula
\begin{equation}
\label{eq:MVT-Asymp}
J_{d,s}(N)= \(\gamma_{d,s}   + o(1)\) N^{2s - s(d)}, 
\end{equation} 
for some  constant $\gamma_{d,s} $, which depends only on $d$ and $s$ (see also the
comment after the formulation of~\cite[Theorem~2]{Bourg}). 

For $h=1$, as a consequence  of~\eqref{eq:MVT}, 
for  any polynomial $f\in \R[Z]$ with~\eqref{eq:Poly} and~\eqref{eq:Approx}  we have 
\begin{equation}
\label{eq:individual}  
|T_{f}(1,N)| \le N^{1+o(1)} \(q^{-1} + N^{-1} + qN^{-d}\)^{1/d(d-1)},  
\end{equation}
as $N\to \infty$, 
see~\cite[Theorem~5]{Bourg}.  The bound~\eqref{eq:individual} also follows if one 
substitutes~\eqref{eq:MVT} in a general inequality 
of Vaughan~\cite[Theorem~5.2]{Vau}. 

However for our purpose we need a version of~\eqref{eq:individual} which applies to the sums $T_{f}(h,N)$ with an arbitrary $h \ne0$. Vinogradov~\cite[Theorem~I, Chapter~VI]{Vin} provides an estimate of $T_{f}(h,N)$. One of his motivations was a question on the distribution of fractional parts of polynomials, see~\cite[Chapter~VIII]{Vin}. However, the  estimate in~\cite{Vin}  is based on an older version of the mean value theorem and thus can now be significantly improved. This is the content of Lemma~\ref{lem:Weyl-h}. In its proof we follow the derivation of~\eqref{eq:individual}  as in the proof~\cite[Theorem~5.2]{Vau}. Also  Lemma~\ref{lem:Weyl-h} is of independent interest and in particular can be used to improve the above-mentioned bound of~\cite[Chapter~VIII]{Vin} on the discrepancy of fractional parts of polynomials. 

\begin{lemma}
\label{lem:Weyl-h}
For  any polynomial $f\in \R[Z]$, satisfying~\eqref{eq:Poly} and~\eqref{eq:Approx}, and for any integer $h \ne 0$ we have 
\begin{equation}
\label{eq:Bound T o(1)} 
|T_{f}(h,N)| \le    N^{1+o(1)} \Delta^{1/d(d-1)}, 
\end{equation}
and 
\begin{equation}
\label{eq:Bound T log} 
T_{f}(h,N) \ll    N \Delta^{1/(d^2-d +2)}  \log N, 
\end{equation}
as $N\to \infty$, where
$$
\Delta =  h q^{-1}+ N^{-1}  +  qN^{-d}+ D N^{-d+1}
$$
with  $D = \gcd(h,q)$.  
\end{lemma}

\begin{proof} We apply the reduction from Weyl sums to~\eqref{eq:MVT}, which is 
given in the proof of~\cite[Theorem~5.2]{Vau}. Hence, we very frequently appeal to 
the definitions and estimates in~\cite[Section~5.2]{Vau}, which we specify to $j=d$.

We now fix some parameter $L\le N$ to be optimised later. 
We remark that we need the condition $L\le N$ in order to apply~\cite[Equaltion~(5.23)]{Vau}. 
Now, for each $x \in \{1,\ldots,L\}$ the number of integers $y \in \{1,\ldots,L\}$ with 
$$
\left\| (d!)^d \alpha_d h(x-y) \right\| \le N^{-d+1}
$$ 
is bounded by  the number integers $y \in \{1,\ldots,L\}$ with 
$$
\left\| (d!)^d a h(x-y)/q \right\| \le N^{-d+1}+ (d!)^d h Lq^{-2} , 
$$ 
which is at most 
$$
R = \((d!)^d D Lq^{-1} +1\)\(2qN^{-d+1} + 2(d!)^d h L q^{-1} + 1\).
$$
This is similar to the definition~\cite[Equation~(5.35)]{Vau}, however adjusted to take into account 
the influence of $h$. 
Indeed  $hy$ may belong to at most 
$2 q( N^{-d+1}+ (d!)^d h Lq^{-2}) + 1$ residue classes modulo $q$. For each residue class,
$y$ is uniquely defined modulo $q/D$.  

We now proceed exactly as  in~\cite[Section~5.2]{Vau} and thus, using~\eqref{eq:MVT} with
$d-1$ instead of $d$ and for  any integer $s \ge 1$, we obtain 
\begin{equation}
\label{eq:TJR/L}  
T_{f}(h,N)^{2s} \ll  N^{d(d-1)/2}  \( \log N\)^{2s}  J_{d-1,s} (2N) R/L  . 
\end{equation}
Hence it remains to choose $L$ to minimise  the ratio $R/L$. 
We have 
$$
R   \ll DL N^{-d+1}  +  Dh L^2q^{-2}+ qN^{-d+1} + h L q^{-1} + 1, 
$$
where we have dropped the term $D L q^{-1} \le h L q^{-1}$. 
Thus
$$
R/L \ll  D N^{-d+1}  + h q^{-1} +  Dh Lq^{-2}+ \(qN^{-d+1}+  1\)/L. 
$$
We now define $L_0$ by the equation 
$$
 Dh L_0q^{-2}= \(qN^{-d+1}+  1\)/L_0
 $$
 and set 
 $$
 L =\min\{\rf{L_0}, N\}.
 $$
 If $L=\rf{L_0}$ then 
\begin{align*}
R/L & \ll  D N^{-d+1}  + h q^{-1} +  \(Dh q^{-2}\(qN^{-d+1}+  1\)\)^{1/2} \\
& \ll  D N^{-d+1}  + h q^{-1} +  D^{1/2} h^{1/2}   q^{-1/2} N^{-(d-1)/2} +  D^{1/2} h^{1/2}   q^{-1}\\
& \ll  D N^{-d+1}  + h q^{-1} +  D^{1/2} h^{1/2}   q^{-1/2} N^{-(d+1)/2}.
\end{align*}
Since the term $D^{1/2} h^{1/2}   q^{-1/2} N^{-(d+1)/2}$ is the geometric mean of the other two terms,
we obtain 
$$
R/L  \ll   D N^{-d+1}  + h q^{-1}.
$$
 If $L=N$ then  
 \begin{align*}
R/L & \ll  D N^{-d+1}  + h q^{-1} +  \(qN^{-d+1}+  1\)/N \\
& \ll  D N^{-d+1}  + h q^{-1} +  qN^{-d} + N^{-1}.
 \end{align*}
 
 Choosing $s = s(d-1)$ and applying~\eqref{eq:MVT} and then also choosing $s = s(d-1)+1$ and applying~\eqref{eq:MVT-Asymp}, 
together with~\eqref{eq:TJR/L},  we derive the desired bounds~\eqref{eq:Bound T o(1)}  and~\eqref{eq:Bound T log}, respectively. 
\end{proof}

In the following remarks we compare the strength of the bounds~\eqref{eq:Bound T o(1)} and~\eqref{eq:Bound T log}  of Lemma~\ref{lem:Weyl-h}, and also compare these bounds with other results. 

\begin{rem} 
The estimate in~\eqref{eq:Bound T log} is valid for a wider range of moduli $q$ at the cost that it is somewhat weaker than~\eqref{eq:Bound T o(1)}. More precisely, the first bound~\eqref{eq:Bound T o(1)} of Lemma~\ref{lem:Weyl-h} is better than the second one~\eqref{eq:Bound T log} unless one of the summands of $\Delta$ are larger than any negative power of $N$.  Large summands occur if $q$ is either small or close to $N^d$, or if $D$ is close to $N^{d-1}$. In these cases~\eqref{eq:Bound T o(1)} becomes trivial and~\eqref{eq:Bound T log} still gives a nontrivial estimate.
\end{rem} 

\begin{rem} 
We note that~\cite[Theorem~9]{Hal} implies a slightly weaker version of~\eqref{eq:Bound T o(1)}. Our improvement is due to the fact that we enter the details of the proof of~\cite[Theorem~5.2]{Vau}. 
\end{rem}

Clearly,  for $d \ge 2$, if $h=1$ and thus $D=1$ the term $ D N^{-d+1}\le N^{-1}$ never dominates  and we recover the bound~\eqref{eq:individual} from~\eqref{eq:Bound T o(1)}. 

Furthermore, if $d \ge 3$ and  $D \le h \le N$, which is a very important case for applications, the bounds of Lemma~\ref{lem:Weyl-h} simplify to
$$
|T_{f}(h,N)| \le    N^{1+ o(1)}  \(h q^{-1}+ N^{-1}  +  qN^{-d}\)^{1/d(d-1)} 
$$
and 
$$
T_{f}(h,N) \ll    N \(h q^{-1}+ N^{-1}  +  qN^{-d}\)^{(1/d^2-d+2)} \log N,
$$
respectively.

Finally, we observe that using the trivial bound $D \le h$ we obtain
\begin{equation}\label{eq:deltawithoutD}
\Delta \le  \(hNq^{-1} +1\) \(N^{-1} + q N^{-d}\).
\end{equation}

\section{Estimates of Weyl sums  with  congruence conditions on the sum of digits function and twisted by  special sequences} 
\label{sec:Thue}

\subsection{Proof of Theorem~\ref{thm:Weyl-cong}}
Following~\cite[Proof of Theorem~6]{BaCoSh} we set $X=2^{r-u}$, where $u\in\{0,\ldots, r\}$  is to be chosen later. For every $n\in \Grs$, write $n= 2^ux +y$ with
$x \in\{0,\ldots,2^{r-u}-1\}$ and $y \in \{0,\ldots,2^u-1\}$. Then
$$
U_{f}(r,\ell,k,m) 
=\sum_{j=0}^m \sum_{x\in\cE_{k-j,m}(r-u)} \sum_{y\in\cE_{j,m}(u)}
e\(\ell f\(2^ux +y\)\).
$$
By the Cauchy inequality, we have
\begin{align*}
 \big|U_{f}(r,\ell,&k,m)\big|^{2} 
   \le  m \sum_{j=0}^{m-1} \# \cE_{k-j,m}(r-u)
\sum_{x=0}^{X-1}\left|\sum_{y\in\cE_{j,m}(u)}\e\(\ell f(2^ux +y)\)\right|^{2}\\
&  =m\sum_{j=0}^{m-1}\# \cE_{k-j,m}(r-u)  \sum_{x=0}^{X-1}
\sum_{y,  z \in\cE_{j,m}(u)}
\e\(\ell f\(2^ux +y\) - \ell f\(2^ux +z\)\)\\
& \le m  \sum_{j=0}^{m-1} \# \cE_{k-j,m}(r-u) 
\(\#\cE_{j,m}(u)X +   \sum_{\substack{y,  z \in\cE_{j,m}(u)\\ y \ne z}} \left|\sum_{x=0}^{X-1}\e\(  F_{y,z} (x) )\) \right|\),
\end{align*} 
where 
\begin{equation}\label{eq:Fdef}
F_{y,z} (Z) = \ell f\(2^u Z+ y \) - \ell f\(2^u Z+ z\) \in \R[Z].
\end{equation}
Hence, using that 
$$
 \sum_{j=0}^{m-1} \# \cE_{k-j,m}(r-u)  \#\cE_{j,m}(u) = \# \cE_{k,m}(r) \le 2^r
\mand 
 \# \cE_{k-j,m}(r-u)  \le 2^{r-u} = X
$$
we infer  (note that from here onwards the estimates are no longer uniform in $m$) 
$$
U_{f}(r,\ell,k,m)^2  \ll 2^r X  +
X\sum_{j=0}^{m-1}    \sum_{\substack{y,  z \in\cE_{j,m}(u)\\ y \ne z}} \left|\sum_{x=0}^{X-1}\e\(  F_{y,z} (x) \) \right|.
$$

Clearly for $y \ne z$ the polynomial $F_{y,z}$
is of the form 
\begin{equation}\label{eq:Fform}
F_{y,z} (Z) = \ell d2^{(d-1)u} (y-z) G_{y,z} (Z)   \in \R[Z], 
\end{equation}
where  $G_{y,z} (Z)   \in \R[Z]$ is of degree $d-1$ and has the leading coefficient $\alpha_d$.

For $y\ne z$, we now recall Lemma~\ref{lem:Weyl-h} with $h \ll \ell 2^{du}$ and we also use the trivial estimate 
$D \le h$. In particular, from~\eqref{eq:Bound T o(1)} (see also~\eqref{eq:deltawithoutD}) 
we obtain 
\begin{equation}
\label{eq:thm2startcong}
\begin{split}
U_{f}(r,\ell,k,m)^2&  \ll  2^r X +      2^{2u}X^{2+o(1)}\(\(\ell 2^{du}Xq^{-1} +1\) \(X^{-1} + q X^{-(d-1)}\)\)^{1/(d-2)(d-1)}  \\
& \le 2^{2r+o(r)}  \(2^{-u} +\(\(\ell 2^{du}Xq^{-1} +1\) \(X^{-1} + q X^{-(d-1)}\)\)^{1/(d-2)(d-1)} \).
 \end{split}
 \end{equation}
 It is now convenient to denote 
 \begin{equation}\label{eq:RUkappa}
 R = 2^r , \qquad  U = 2^u, \qquad  \kappa = \frac{1}{(d-1)(d-2)}.
 \end{equation}
 Then we obtain 
 \begin{equation}\label{eq:UwithRUkappa}
 \left|U_{f}(r,\ell,k,m)\right|^{2}   \le     R^{2+o(1)} \(U^{-1} +   \(\(  \ell R U^{d-1} q^{-1} +1\) \(R^{-1} U +   R^{-(d-1)} U^{d-1} q\)\)^{\kappa} \).
\end{equation}
Expanding the product and pulling in the exponent we get
$$
 \left|U_{f}(r,\ell,k,m)\right|^{2}   \le       R^{2+o(1)}  \Delta, 
$$
 where 
 \begin{equation}\label{eq:defDeltaKappa}
 \begin{split}
  \Delta  = 
\ell^\kappa U^ {d \kappa}q^{-\kappa}   &
+ \ell^\kappa R^{-(d-2)\kappa}U^{2(d-1)\kappa} + R^{-\kappa} U^{\kappa} \\
&\qquad \qquad  
+ R^{- (d-1)\kappa}U^{(d-1)\kappa}q^{\kappa} +U^{-1}  .
\end{split}
\end{equation}
We now apply Lemma~\ref{lem: Optim} with $I = 4$, $J=1$, $Z=U$, $Z_1=0$, $Z_2=R$ and parameters
$$
(A_i,a_i)_{i=1}^4   = \(\(\ell^{\kappa}q^{-\kappa}, d \kappa\), \( \ell^{\kappa}R^{-(d-2)\kappa}, 2(d-1)\kappa\), \(R^{-\kappa}, \kappa\), \(R^{- (d-1)\kappa} q^{\kappa}, (d-1)\kappa\)\)
$$
and 
$$
(B_1,b_1) = (1,1). 
$$
Hence according to Lemma~\ref{lem: Optim} there is a choice of $U$ such that 
\begin{equation}\label{eq:defDeltaKappaEstimate}
\begin{split}
\Delta   \ll \ell^{\kappa/(d\kappa+1)}q^{-\kappa/(d\kappa+1)}& + \ell^{\kappa/(2(d-1)\kappa+1)}R^{-(d-2)\kappa/(2(d-1)\kappa+1)} \\
& \qquad  + R^{-\kappa/(\kappa+1)} +R^{- (d-1)\kappa/((d-1)\kappa+1)} q^{\kappa/((d-1)\kappa+1)}
\end{split}
\end{equation}
(note that the term $Z_2^{-b_1}=R^{-1}$ never dominates so we have omitted it). It is clear from the proof of Lemma~\ref{lem: Optim} in~\cite{GrKol} that this choice of $U$ is optimal.
Inserting $R=2^r$, the result now follows.

\subsection{Proof of Theorem~\ref{thm:Weyl-congLog}}

The proof runs along similar lines.

In the same way as we have derived~\eqref{eq:thm2startcong} we get, using~\eqref{eq:Bound T log} instead of~\eqref{eq:Bound T o(1)} and observing that $s\le r$, that
\begin{align*}
U_{f}(r,\ell,k,m)^2
& \ll r^22^{2r}  \(2^{-u} +   \(\(\ell 2^{du}Xq^{-1} +1\) \(X^{-1} + q X^{-(d-1)}\)\)^{1/(d-2)(d-1)} \).
 \end{align*}

It is now convenient to denote 
 \begin{equation}\label{eq:RUlambda}
 R = 2^r , \qquad  U = 2^u, \qquad  \lambda = \frac{1}{d^2-3d+4}.
 \end{equation}
 Then we obtain 
$$
U_{f}(r,\ell,k,m)^2   \ll     r^2R^2 \(U^{-1} +   \(\(  \ell R U^{d-1} q^{-1} +1\) \(R^{-1} U +   R^{-(d-1)} U^{d-1} q\)\)^{\lambda} \).
$$ 
We expand the product and pull in the exponent to get
$$
U_{f}(r,\ell,k,m)^2  \ll       r^2R  \binom{r}{s} \Delta, 
$$
 where 
 \begin{equation}\label{eq:defDeltaLambda}
  \Delta  = 
\ell^\lambda U^ {d \lambda}q^{-\lambda}  
+ \ell^\lambda R^{-(d-2)\lambda}U^{2(d-1)\lambda} + R^{-\lambda} U^{\lambda}  
+ R^{- (d-1)\lambda}U^{(d-1)\lambda}q^{\lambda} +U^{-1}  .
\end{equation} 
We can now apply Lemma~\ref{lem: Optim} in the same way as in the proof of Theorem~\ref{thm:Weyl-cong} to
obtain that there is a choice of $U$ such that 
\begin{equation}\label{eq:defDeltaLambdaEstimate}
\begin{split}
\Delta   \ll \ell^{\lambda/(d\lambda+1)}q^{-\lambda/(d\lambda+1)}& + \ell^{\lambda/(2(d-1)\lambda+1)}R^{-(d-2)\lambda/(2(d-1)\lambda+1)} \\
& \qquad  + R^{-\lambda/(\lambda+1)} +R^{- (d-1)\lambda/((d-1)\lambda+1)} q^{\lambda/((d-1)\lambda+1)}
\end{split}
\end{equation}
(again the term $Z_2^{-b_1}=R^{-1}$ never dominates so we have omitted it). Inserting $R=2^r$, the result follows.

\subsection{Proof of Theorem~\ref{thm:Weyl-RS}}
\label{seq:Weyl-RS}

Observe that 
\begin{equation}\label{eq:RSsplit}
\RS_f(r,\ell) = 2  R_f(r,\ell)  - \sum_{n<2^r} \e(\ell f(n)),
\end{equation}
where 
$$
R_f(r,\ell) = \sum_{n \in \cH_0(r) }  \e(\ell f(n))
$$
is the Weyl sum over the domain $\cH_0(r)$ with
$$
\cH_k(r) = \{ n<2^r :~ \chi_{11} (n) \equiv k \pmod{2} \}.
$$
Using Lemma~\ref{lem:Weyl-h} the classical Weyl sum in the second summand on the right hand side of~\eqref{eq:RSsplit} can be estimated better than our claimed bound on $\RS_f(r,\ell)$. Thus it remains to estimate the sum $R_f(r,\ell)$. This can be achieved by splitting the sum according to the values of the digits $a_{u-1}(n)$ and $a_u(n)$ of the binary expansion of $n$ in~\eqref{eq:binrep}. Indeed, we write
\begin{equation}\label{eq:u2split}
R_f(r,\ell) = \sum_{c_0,c_1 \in \{0,1\}} R_f(r,\ell,c_0,c_1)
\end{equation}
with
\begin{equation}\label{eq:U2c0c1}
R_f(r,\ell,c_0,c_1) = \sum_{\substack{n \in \cH_0(r) \\(a_u(n),a_{u-1}(n))=(c_0,c_1)}}  \e(\ell f(n) \qquad(c_0,c_1 \in \{0,1\}).
\end{equation}
As before we set $X=2^{r-u}$ where $u\in\{0,\ldots, r\}$  is to be chosen later. For every $n\in \Grs$, write $n= 2^ux +y$ with $x \in\{0,\ldots,2^{r-u}-1\}$ and $y \in \{0,\ldots,2^u-1\}$. 
Thus, for all $c_0,c_1 \in \{0,1\}$ the sum in~\eqref{eq:U2c0c1} can be written as
$$
R_f(r,\ell,c_0,c_1) =
\sum_{j=0}^1\sum_{\substack{x\in\cH(r-u,j) \\ x\equiv c_0\pmod{2}}}
\sum_{y\in\cH(u,j+c_0c_1)}   \e(\ell f(2^ux+c_12^{u-1}+y)).
$$
Applying the Cauchy inequality yields
\begin{align*}
 \big|R_f(r,\ell,&c_0,c_1)\big|^{2} 
    \le  2 X \sum_{j=0}^{1} 
\sum_{x=0}^{X-1}\left|\sum_{y\in\cH(u,j+c_0c_1)}\e\(\ell f(2^ux +c_12^{u-1}+y)\)\right|^{2}\\
&  
=2 X \sum_{j=0}^{1} 
\sum_{x=0}^{X-1} 
\sum_{y,  z \in\cH(u,j+c_0c_1)}
\e\(\ell f\(2^ux +c_12^{u-1}+y\) - f\(2^ux +c_12^{u-1} +z\)\)\\
& \le 2 X \sum_{j=0}^{1} 
\(X2^{u} +   \sum_{\substack{y,  z \in\cH(u,j+c_0c_1)\\ y \ne z}} \left|\sum_{x=0}^{X-1}\e\(  F_{y,z} (x) )\) \right|\),
\end{align*} 
where 
$$
F_{y,z} (Z) = \ell f\(2^u Z+c_12^{u-1}+ y \) - \ell f\(2^u Z+c_12^{u-1}+ z\) \in \R[Z].
$$
Hence, 
$$
R_f(r,\ell,c_0,c_1)^2  \ll 2^r X  +
X\sum_{j=0}^1    \sum_{\substack{y,  z\cH(u,j+c_0c_1)\\ y \ne z}} \left|\sum_{x=0}^{X-1}\e\(  F_{y,z} (x) \) \right|.
$$

For $y \ne z$ the polynomial $F_{y,z}$ is again of the form~\eqref{eq:Fform}.
For $y\ne z$, we now recall Lemma~\ref{lem:Weyl-h} with $h \ll \ell 2^{du}$ and we also use the trivial estimate 
$D \le h$. In particular, from~\eqref{eq:Bound T o(1)} (see also~\eqref{eq:deltawithoutD}) 
we obtain 
\begin{equation*}
\begin{split}
R_f(r,\ell,c_0,c_1)^2&  \ll  2^r X +      2^{2u}X^{2+o(1)}\(\(\ell 2^{du}Xq^{-1} +1\) \(X^{-1} + q X^{-(d-1)}\)\)^{1/(d-2)(d-1)}. 
 \end{split}
 \end{equation*}
 Because the right hand side no longer depends on $c_0$ and $c_1$, according to~\eqref{eq:u2split} we gain
\begin{equation}\label{eq:u2final}
\begin{split}
R_f(r,\ell)^2&  \ll  2^r X +      2^{2u}X^{2+o(1)}\(\(\ell 2^{du}Xq^{-1} +1\) \(X^{-1} + q X^{-(d-1)}\)\)^{1/(d-2)(d-1)}.
 \end{split}
 \end{equation}
 Since the right hand side of~\eqref{eq:u2final} coincides with the right hand side of the first
estimate in~\eqref{eq:thm2startcong}, the result follows by {\em verbatim} repeating the proof of Theorem~\ref{thm:Weyl-cong} from~\eqref{eq:thm2startcong} onwards.

\subsection{Proof of Theorem~\ref{thm:Weyl-sigmasigma}}
\label{sec::Weyl-sigmasigma}

As in Section~\ref{seq:Weyl-RS} we see that it is enough to estimate the sum 
$$
V_f(r,\ell) = \sum_{n \in \cF(r) }  \e(\ell f(n)),
$$
where
$$
\cF(r) = \{ n<2^r :~ \sigma(n) + \sigma(n+1) \equiv 0 \pmod{2} \}.
$$
Indeed, we have
$$
W_f(r,\ell) =  2 V_f(r,\ell)  - \sum_{n<2^r} \e(\ell f(n))
$$
and the last classical Weyl sum can be estimated better than our claimed bounds on $W_f(r,\ell)$, see Lemma~\ref{lem:Weyl-h}. 

Again we set $X=2^{r-u}$ where $u\in\{0,\ldots, r\}$  is to be chosen later. For every $n\in \Grs$, write $n= 2^ux +y$ with $x \in\{0,\ldots,2^{r-u}-1\}$ and $y \in \{0,\ldots,2^u-1\}$. For convenience, we set
$$
\widetilde \cF(r) = \{ n<2^r-1 :~ \sigma(n) + \sigma(n+1) \equiv 0 \pmod{2} \}
$$
and
$$
\widetilde  V_f(r,\ell) = \sum_{\substack{n \in \cF(r) \\ n \not\equiv -1\pmod{2^u} }}  \e(\ell f(n)).
$$  
Then 
\begin{equation}\label{eq:VVprime}
V_f(r,\ell) = \widetilde  V_f(r,\ell) + \sum_{\substack{n \in \cF(r) \\ n \equiv -1\pmod{2^u} }}  \e(\ell f(n)).
\end{equation}
We  estimate the sum on the left trivially by $X$. 
Thus it remains to deal with the sum $\widetilde  V_f(r,\ell)$. We have
$$
\widetilde  V_f(r,\ell) = \sum_{x=0}^{X-1} \sum_{y\in\widetilde \cF(u)} \e\(\ell f\(2^ux +y\)\).
$$
By the Cauchy inequality, we derive
\begin{align*}
 \left|\widetilde  V_f(r,\ell)\right|^{2} 
&   \le X
\sum_{x=0}^{X-1}\left|\sum_{y\in\widetilde \cF(u)}\e\(\ell f(2^ux +y)\)\right|^{2}\\
&  =X  \sum_{x=0}^{X-1} \sum_{y,  z \in\widetilde \cF(u)}
\e\(\ell f\(2^ux +y\) - f\(2^ux +z\)\)\\
& \ll X \(2^uX +   \sum_{\substack{y,  z =0\\ y \ne z}}^{2^u-1} \left|\sum_{x=0}^{X-1}\e\(  F_{y,z} (x) )\) \right|\),
\end{align*}
with $F(Z)$  as in~\eqref{eq:Fdef}.

Again, we can go on along the lines of the proof of Theorem~\ref{thm:Weyl-cong} 
until we arrive at~\eqref{eq:UwithRUkappa}. Thereby, using the notation from~\eqref{eq:RUkappa}, we gain
\begin{equation}\label{eq:VprimeLast}
 \left|\widetilde  V_f(r,\ell)\right|^{2}   \le     R^{2+o(1)} \(U^{-1} +   \(\(  \ell R U^{d-1} q^{-1} +1\) \(R^{-1} U +  R^{-(d-1)} U^{d-1} q\)\)^{\kappa} \).
\end{equation}
But if we keep in mind that the second sum on the right hand side of~\eqref{eq:VVprime} 
can be trivially estimated by $X=RU^{-1}$, squaring~\eqref{eq:VVprime} and pulling the square in, inserting~\eqref{eq:VprimeLast} yields
$$
 \left|V_f(r,\ell)\right|^{2}   \le     R^{2+o(1)} \(U^{-1} +   \(\(  \ell R U^{d-1} q^{-1} +1\) \(R^{-1} U +  R^{-(d-1)} U^{d-1} q\)\)^{\kappa} \).
$$
The proof is now finished in the same way as the proof of Theorem~\ref{thm:Weyl-cong} 

\section{Estimates of Weyl sums with a fixed sum of digits function}
\label{sec:sparseproof}

\subsection{Proof of Theorem~\ref{thm:Weyl-sparse}}
Put $X=2^{r-u}$ where $u\in\{0,\ldots,r\}$  is to be chosen later.
For every $n\in \Grs$, write $n= 2^ux +y$ with
$x \in\{0,\ldots,2^{r-u}-1\}$ and $y \in \{0,\ldots,2^u-1\}$.
Then
$$S_{f}(r, \ell, s)
=\sum_{j=0}^s \sum_{x\in\cG_{s-j} (r-u)} \, \sum_{y\in\cG_j(u)}\, 
\e\(\ell f\(2^ux +y\)\).$$
By applying the Cauchy inequality twice (first to the sum over~$j$ and then to the sum over~$x$), we have 
\begin{align*}
 \left|S_{f}(r, \ell, s)\right|^{2} 
&   \le  (s+1) \sum_{j=0}^s\binom{r-u}{s-j} 
\sum_{x=0}^{X-1}\left|\sum_{y\in\cG_j(u)}\, \e\(\ell f(2^ux +y)\)\right|^{2}\\
&  =(s+1)\sum_{j=0}^s\binom{r-u}{s-j}  \sum_{x=0}^{X-1}\\
&   \qquad   \qquad \sum_{y,  z \in\cG_j(u)}\,
\e\(\ell f\(2^ux +y\) - f\(2^ux +z\)\)\\
& \ll s \sum_{j=0}^s \binom{r-u}{s-j}  \(\binom{u}{j}X +   \sum_{\substack{y,  z \in\cG_j(u)\\ y \ne z}} \left|\sum_{x=0}^{X-1}\e\(  F_{y,z} (x) \) \right|\), 
\end{align*}
with $F_{y,z}(Z)$ as in~\eqref{eq:Fdef}.
Hence,
\begin{align*}
S_{f}(r,\ell, s)^2 & \ll sX \sum_{j=0}^s \binom{r-u}{s-j}  \binom{u}{s}
 \\
& \qquad \qquad \qquad +
s\sum_{j=0}^s \binom{r-u}{s-j}    \sum_{\substack{y,  z \in\cG_j(u)\\ y \ne z}} \left|\sum_{x=0}^{X-1}\e\(  F_{y,z} (x) \) \right| 
\\
& = sX \binom{r}{s} + s\sum_{j=0}^s \binom{r-u}{s-j}    \sum_{\substack{y,  z \in\cG_j(u)\\ y \ne z}} \left|\sum_{x=0}^{X-1}\e\(  F_{y,z} (x) \) \right|.
\end{align*}

For $y \ne z$ the polynomial $F_{y,z}$ is of the form~\eqref{eq:Fform}.
 For $y\ne z$, we now recall Lemma~\ref{lem:Weyl-h} with $h \ll \ell 2^{du}$ and we also use the trivial estimate 
$D \le h$. In particular, from~\eqref{eq:Bound T o(1)} (see also~\eqref{eq:deltawithoutD}) 
we obtain 
\begin{align*}
S_{f}(r, \ell, s)^2  \ll   sX \binom{r}{s} + &sX^{(1+o(1))} \sum_{j=0}^s  \binom{r-u}{s-j}   \binom{u}{j}^2\\  
& \quad \(\(\ell 2^{du}Xq^{-1} +1\) \(X^{-1} + q X^{-(d-1)}\)\)^{1/(d-2)(d-1)}  .
 \end{align*}
We now use that 
 $$
 \sum_{j=0}^s   \binom{r-u}{s-j}   \binom{u}{j}^2\le 2^u \sum_{j=0}^s   \binom{r-u}{s-j}   \binom{u}{j} = 2^u  \binom{r}{s}
$$ 
to derive
\begin{equation}
\label{eq:thm2start}
\begin{split}
S_{f}(r, \ell, s)^2&  \ll   s X \binom{r}{s}\\
& \qquad  + s2^{u}  X^{1+o(1)} \binom{r}{s}  \(\(\ell 2^{du}Xq^{-1} +1\) \(X^{-1} + q X^{-(d-1)}\)\)^{1/(d-2)(d-1)}  \\
& \le 2^{r+o(r)}   \binom{r}{s}\\
& \qquad \(2^{-u} +   \(\(\ell 2^{r+ (d-1)u} q^{-1} +1\) \(2^{-r+u} + q 2^{-(d-1)(r-u)}\)\)^{1/(d-2)(d-1)} \)
 \end{split}
 \end{equation}
 (note that the factor $s$ has been absorbed in the term $2^{o(r)}$). 
 We again use the notation~\eqref{eq:RUkappa} and obtain 
 \begin{align*}
 \left|S_{f}(r, \ell, s)\right|^{2} &  \le     R^{1+o(1)}   \binom{r}{s}\\
& \qquad \(U^{-1} +   \(\(  \ell R U^{d-1} q^{-1} +1\) \(R^{-1} U +   R^{-(d-1)} U^{d-1} q\)\)^{\kappa} \).
 \end{align*}  
We now expand the product and pull in the exponent to get
$$
 \left|S_{f}(r, \ell, s)\right|^{2}   \le       R^{1+o(1)}   \binom{r}{s} \Delta, 
$$ 
with $\Delta$ given by~\eqref{eq:defDeltaKappa}. Thus the result follows from~\eqref{eq:defDeltaKappaEstimate}.

\subsection{Proof of Theorem~\ref{thm:Weyl-sparseLog}}

The proof runs along similar lines.

In the same way as we derived~\eqref{eq:thm2start} we get, using~\eqref{eq:Bound T log} instead of~\eqref{eq:Bound T o(1)} and observing that $s\le r$, that
\begin{align*}
S_{f}(r, \ell, s)^2
& \ll r^22^{r}   \binom{r}{s} \(2^{-u} +   \(\(\ell 2^{r+ (d-1)u} q^{-1} +1\) \(2^{-r+u} + q 2^{-(d-1)(r-u)}\)\)^{1/(d^2-3d+4)} \).
 \end{align*}
Using the notation~\eqref{eq:RUlambda}
this becomes 
$$
S_{f}(r, \ell, s)^2   \ll     r^2R   \binom{r}{s} \(U^{-1} +   \(\(  \ell R U^{d-1} q^{-1} +1\) \(R^{-1} U +   R^{-(d-1)} U^{d-1} q\)\)^{\lambda} \).
$$
We now expand the product and pull in the exponent to get
$$
S_{f}(r, \ell, s)^2  \ll       r^2R  \binom{r}{s} \Delta, 
$$
with $\Delta$ as in~\eqref{eq:defDeltaLambda}. Thus the result follows from~\eqref{eq:defDeltaLambdaEstimate}.

\section{Proof of the equidistribution results}
\label{sec:equi}

\subsection{Discrepancy and exponential sums} 
Our goal is to estimate the discrepancies $D_f(\cE_{k,m}(r))$ and $D_f(\Grs)$ provided that the leading coefficient $\alpha_d$ of the polynomial $f$ is of Diophantine type $1$. To  
achieve this we recall  the Erd\H{o}s--Tur{\'a}n--Koksma inequality (see, for example,~\cite[Theorem~1.21]{DT:97}).

We recall that $\cL(\cI)$ denotes the   Lebesgue measure of 
an interval $\cI \subseteq[0,1]$.

\begin{lemma}
\label{lem:ET small int}
Let $y_1, \ldots, y_N$ be a sequence of $N$ points of the interval $[0,1)$.
Then for any integer $L\ge 1$, and any interval $\cI \subseteq [0,1)$,
we have
$$
|\# \{n \in\N: 1\le n\le N,\,~y_n  \in \cI\} - N\cL(\cI)|
\ll \frac{N}{L} +  \sum_{\ell=1}^L  \frac{1}{\ell} 
\left|\sum_{n=1}^N \e\(\ell y_n\)\right|.
$$
\end{lemma}

\subsection{Proof of  Theorem~\ref{thm:cong-Equidistr}} 
By Lemma~\ref{lem:ET small int}, for any integer $L \ge 1$ we have
\begin{equation}\label{eq:ETKE}
D_f(\cE_{k,m}(r)) \ll  
\frac{1}{L} +  \frac{1}{\#\cE_{k,m}(r)}  \sum_{\ell=1}^L \frac{1}{\ell} \left| U_{f}(r,\ell,k,m)\right|  .
\end{equation}
We now invoke  Corollary~\ref{cor:Weyl-cong-Dioph-1} and then choose $L$ in a way that makes the resulting bound optimal. In particular, using~\eqref{eq:cor2bdcong} and observing that 
$$
\# \cE_{k,m}(r)= \frac{2^r}{m} + O(2^{r(1-\delta)})
$$ 
for some explicitly computable constant $\delta >0$ (see~\cite{Gelfond:68})
we gain
\begin{align*}
 \frac{1}{\#\cE_{k,m}(r)} & \sum_{\ell=1}^L \frac{1}{\ell} \left| U_{f}(r,\ell,k,m)\right|  \\
& \qquad \quad 
\le
2^{o(r)} \(\(\frac{L^{\eta_2(d)}}{2^{\zeta_3(d)r}}\)^{\eta_1(d)/(\eta_1(d)+\eta_2(d))} + \frac{L^{\vartheta(d)}}{2^{\zeta_1(d)r}}  + 2^{-\zeta_2(d)r}\)^{1/2}.
\end{align*} 
Pulling in the exponent $1/2$ and inserting this in~\eqref{eq:ETKE} yields 
\begin{equation}\label{eq:disc2E}
\begin{split}
D_f(\cE_{k,m}(r))& \ll  2^{o(r)} \\
& \qquad   \(\(\frac{L^{\eta_2(d)}}{2^{\zeta_3(d)r}}\)^{\eta_1(d)/(2(\eta_1(d)+\eta_2(d)))} + \frac{L^{\vartheta(d)/2}}{2^{\zeta_1(d)r/2}}  + 2^{-\zeta_2(d)r/2}\) + \frac{1}{L} .
\end{split} 
\end{equation} 
We now use Lemma~\ref{lem: Optim} to optimise for $L$. Indeed, we choose 
$$
Z_1=0\mand Z_2=2^r
$$ 
and
\begin{align*}
(A_1,a_1) &= \(2^{-\frac{\zeta_3(d)\eta_1(d)}{2(\eta_1(d)+\eta_2(d))}r}, \frac{\eta_1(d)\eta_2(d)}{2(\eta_1(d)+\eta_2(d))}\),
\\
(A_2,a_2) &= \(2^{- \zeta_1(d)r/2}, \vartheta(d)/2 \), \\
(B_1,b_1) &= (1,1).
\end{align*} 
With these choices we apply Lemma~\ref{lem: Optim} to all summands on the right hand side of~\eqref{eq:disc2E} which contain $L$ (leaving the remaining summand unchanged) and obtain
$$
D_f(\cE_{k,m}(r))\ll  
2^{-\frac{d-1}{4d^2-8d+7}r+o(r)} 
+
2^{-\frac{d-2}{2d^2-2d+1}r+o(r)} 
+
2^{-\frac{1}{2d^2-6d+6}r + o(r)}
+ 
2^{-r}.
$$ 
Direct calculation shows that the terms $2^{-\frac{d-1}{4d^2-8d+7}r+o(r)}$ and $2^{-r}$ never dominate. Thus the estimate in Theorem~\ref{thm:cong-Equidistr} is established.

\subsection{Proof of  Theorem~\ref{thm:Sparse-Equidistr}} 
The proof is similar to the proof of Theorem~\ref{thm:cong-Equidistr}. Again by Lemma~\ref{lem:ET small int}, for any integer $L \ge 1$ we have
$$
D_f( \Grs) \ll  
\frac{1}{L} +  \frac{1}{\# \Grs}  \sum_{\ell=1}^L \frac{1}{\ell} \left| S_{f}(r, \ell, s)\right|  .
$$
This time we invoke  Corollary~\ref{cor:Weyl-sparse-Dioph-1} and then choose $L$ in a way that makes the resulting bound optimal. In particular, using~\eqref{eq:cor2bd} and observing that 
$$
\# \Grs={r \choose s}
$$ 
we gain as in the proof of Theorem~\ref{thm:cong-Equidistr} that
\begin{equation}\label{eq:disc2G}
\begin{split}
D_f( \Grs)& \ll  2^{(1/2+o(1))r}  \binom{r}{s}^{-1/2}\\
& \qquad  \quad \(\(\frac{L^{\eta_2(d)}}{2^{\zeta_3(d)r}}\)^{\eta_1(d)/(2(\eta_1(d)+\eta_2(d)))} + \frac{L^{\vartheta(d)/2}}{2^{\zeta_1(d)r/2}}  + 2^{-\zeta_2(d)r/2}\) + \frac{1}{L} .
\end{split} 
\end{equation}  
We now use Lemma~\ref{lem: Optim} to optimise for $L$. Indeed, we choose 
$$Z_1=0\mand Z_2={r \choose s}$$ 
and
\begin{align*}
(A_1,a_1) &= \({r \choose s}^{-1/2}2^{\frac12\(1- \frac{\zeta_3(d)\eta_1(d)}{\eta_1(d)+\eta_2(d)}\)r}, \frac{\eta_1(d)\eta_2(d)}{2(\eta_1(d)+\eta_2(d))}\),
\\
(A_2,a_2) &= \({r \choose s}^{-1/2}2^{\frac12\(1- \zeta_1(d)\)r}, \vartheta(d)/2 \), \\
(B_1,b_1) &= (1,1).
\end{align*} 
With these choices we apply Lemma~\ref{lem: Optim} to all summands on the right hand side of~\eqref{eq:disc2G} which contain $L$ (leaving the remaining summand unchanged) and obtain
\begin{align*}
D_f&( \Grs) \ll   \\
&\binom{r}{s}^{-\frac{2 d^2-4 d+3}{4d^2-8d+7}}
2^{\frac{2 d^2-5 d+4}{4d^2-8d+7}r+o(r)} 
+
\binom{r}{s}^{-\frac{d(d-1)}{2d(d-1)+1}}
2^{\frac{d(d-2)+2}{2 d(d-1)+1}r+o(r)} 
+
{r \choose s}^{-\frac{1}{2}}
2^{\frac{(d-2) (d-1)}{2 (d^2-3d+3)} r + o(r)}
\end{align*}  
(the term ${r \choose s}^{-1}$ corresponding to the pair $(B_1,b_1)$ never dominates and is therefore omitted). This estimate establishes Theorem~\ref{thm:Sparse-Equidistr}.

\section{Comments}

We remark that out treatment of the sums $U_f(r,\ell)$  and $V_f(r,\ell)$ in 
Sections~\ref{seq:Weyl-RS} and~\ref{sec::Weyl-sigmasigma} respectively can easily be extended to sums with  with more general congruence conditions.  In particular, our methods allow to improve the bounds of~\cite[Theorem~1]{Emi3} 
as well as of~\cite[Theorem~3.4]{ThTi} (and of the more general~\cite[Theorem~2.2]{PfTh}) for values of the leading coefficient $\alpha_d$ corresponding to the minor arcs in Waring's problem. This leads to an improvement on the number of summands in the versions of Waring's problem proved in these papers.

\section*{Acknowledgements}
During preparation this work, the first author   was  supported in part  by ARC Grant DP170100786
and   the second author  by the FWF Grants I~3466 and P~29910.

\end{document}